\documentclass[12pt]{amsart}
\usepackage{amsmath}
\usepackage{amssymb}
\usepackage{amscd}
\usepackage{graphicx}
\usepackage{bm}
\newtheorem{theorem}{Theorem}[section]
\newtheorem{lemma}[theorem]{Lemma}
\newtheorem{proposition}[theorem]{Proposition}

\theoremstyle{definition}
\newtheorem{definition}[theorem]{Definition}

\numberwithin{equation}{section}

\newtheorem{remark}[theorem]{Remark}

\title{Mixed Frobenius Structure and Local Quantum Cohomology}
\author{Yukiko Konishi}
\address{ 
Department of Mathematics, 
Kyoto University,
Kyoto 606-8502, Japan 
}
\email{konishi@math.kyoto-u.ac.jp}
\author{Satoshi Minabe}
\address{Department of Mathematics, 
Tokyo Denki University, 120-8551 Tokyo,  Japan}
\email{minabe@mail.dendai.ac.jp}
\begin{document}
\maketitle
\begin{abstract}
In a previous paper, the authors introduced the notion of mixed Frobenius structure (MFS) 
as a generalization of the structure of a Frobenius manifold. 
Roughly speaking, the MFS is defined by replacing a metric of the Frobenius manifold
with a filtration on the tangent bundle equipped with metrics on its graded quotients.
The purpose of the current paper is to 
construct a MFS on the cohomology
of a smooth projective variety whose multiplication is the non-equivariant limit 
of the quantum product twisted by  
a concave vector bundle.
We show that such a MFS is naturally obtained as 
the non-equivariant limit of the Frobenius structure
in the equivariant setting. 
\end{abstract}
\section{Introduction}
We continue our study on mixed Frobenius structure and local quantum cohomology initiated in \cite{KonishiMinabe12}.

\subsection{A mixed Frobenius algebra}
Let $K$ be a field.  A finite-dimensional  associative
commutative $K$-algebra $A$ 
equipped with a non-degenerate bilinear form $g$ 
(called a {\it metric}) 
is called a Frobenius algebra if $g$ is 
invariant under the product, i.e.,
$g(xy,z)=g(x,yz)$ holds for any $x,y,z\in A$. 

In \cite{KonishiMinabe12}, the following generalization of the Frobenius algebra was introduced.
Let $A$ be a $K$-algebra as above. By definition, a Frobenius filtration $(I_{\bullet}, g_{\bullet})$ on $A$
consists of an exhaustive increasing filtration $I_{\bullet}$ by ideals
and $A$-invariant metrics $g_{\bullet}$ on its graded quotients (Definition \ref{def:frobenius-filtration}).
We call an algebra with a 
Frobenius filtration
a mixed Frobenius algebra. 
If the filtration is trivial,
this is nothing but the notion of 
Frobenius algebra. 
We show that
any algebra over an algebraically closed field 
admits a Frobenius filtration (Theorem \ref{thm:frobenius}).
This is in contrast to the fact that not all algebras 
admit invariant metrics.  

One of main results of this paper is to show
that
a mixed Frobenius $K$-algebra appears in the limit as $\lambda\to 0$
of a ``Frobenius algebra over $K[\lambda]$"
(\S \ref{section:MFA2}).
The precise statement is as follows.
Let $H^{\lambda}_K$ be a free $K[\lambda]$-module of finite rank
equipped with a symmetric $K[\lambda]$-bilinear form $g^{\lambda}:H^{\lambda}_K\times H^{\lambda}_K\to K[\lambda,\lambda^{-1}]$.
If $g^{\lambda}$ is unimodular over $K[\lambda,\lambda^{-1}]$,
then it defines 
on the $K$-vector space $H_K:=H^{\lambda}_K/\lambda H^{\lambda}_K$
an exhaustive increasing filtration by  subspaces 
and metrics on its graded quotients (Lemma \ref{induced-nondegenerate-filtration}). 
We call such a pair a {\it nondegenerate filtration}. 
Moreover if $H^{\lambda}_K$ is equipped with a $K[\lambda]$-algebra structure
with respect to which $g^{\lambda}$ is invariant,
then 
the nondegenerate filtration is a Frobenius filtration on $H_K$
with respect to the induced multiplication
(Theorem \ref{thm-MFA}).
This construction is a generalization of the nilpotent construction
\cite[\S 3.1]{KonishiMinabe12} (cf. \S \ref{nilpotent-construction}).

\subsection{A mixed Frobenius structure}
A Saito structure (without a metric)\footnote{
In this article, we call a Saito structure without a metric
{\it  a Saito structure} for short.} 
on a complex manifold $M$
\cite[\S VII.1]{Sabbah}
is a triple consisting of 
a torsion-free flat affine connection $\nabla$, 
a symmetric Higgs field $\Phi:T_M\to \mathrm{End}\, T_{M}$ 
and a vector field $E$ called an Euler vector field
satisfying certain compatibility conditions (see Definition \ref{def-Saito-structure}).
A symmetric Higgs field gives rise to a fiberwise commutative associative multiplication 
$\circ$ on $T_{M}$. 
If a Saito structure $(\nabla, \Phi, E)$ on $M$ is further equipped 
with a $\circ$-invariant metric $g$ on $T_M$ compatible with the other data,
then the Saito structure $(\nabla, \Phi, E)$ with the metric $g$ is equivalent to 
a Frobenius manifold structure on $M$ \cite{Dubrovin}.

Now we introduce the notion of the mixed Frobenius structure
which generalizes the Frobenius manifold structure. 
The idea is to replace a $\circ$-invariant metric $g$ with a Frobenius filtration $(I_{\bullet}, g_{\bullet})$.
Namely, we define a mixed Frobenius structure (MFS) on a manifold $M$ 
to be a Saito structure $(\nabla, \Phi, E)$ on $M$
together with a nondegenerate filtration $(I_{\bullet}, g_{\bullet})$
on the tangent bundle $T_M$
subjecting to various compatibility conditions (Definition \ref{def-MFS}).
In particular, it is required that $(T_M, \circ, I_{\bullet}, g_{\bullet})$
is a mixed Frobenius algebra.
We arrived at this notion through our study on 
local mirror symmetry \cite{CKYZ}. 
For detail about the motivation, we refer to \cite[\S 1.3]{KonishiMinabe12}.
Notice that we slightly modify the definition of the MFS
from \cite{KonishiMinabe12} (cf.\,Remark \ref{rem:definitions}).

For  the  application to the local quantum cohomology,
it is necessary to consider 
a formal and logarithmic version of MFS.
Let $K$ be a subfield of $\mathbb{C}$.
Let $R=K[[t_1,\ldots,t_n,q_1,\ldots,q_m]]$
and $M={\rm Spf}\, R$ 
be the formal completion of $K^{n+m}=\mathrm{Spec}\,K[t,q]$
at the origin. We consider the logarithmic structure on $M$
defined by the divisor $\{ q_1\cdots q_m=0\}$ on $K^{n+m}$.
We denote by $M^{\dagger}$ the resulting logarithmic formal scheme. 
We define a formal (and logarithmic) MFS
on $M^{\dagger}$ in \S \ref{sec:formal-MFS}.
As in the case of mixed Frobenius algebra, we show that 
a formal MFS on $M^{\dagger}$ is obtained in 
the limit as $\lambda\to 0$ of a ``formal Frobenius structure
over $K[\lambda]$'' on $M^{\dagger}$ (Proposition \ref{construction2}).

\subsection{MFS from local quantum cohomology}
Let $X$ be a smooth complex projective variety and
let $H_{\mathbb{C}}:=H^{\rm even} (X, \mathbb{C})$.
We choose a nef basis $\{\phi_{1},\ldots,\phi_p\}$
of $H^2(X,\mathbb{Z})$  
and extend it to a homogenious basis 
$\{\phi_0=1,\phi_1,\ldots,\phi_p,\phi_{p+1},\ldots,\phi_s\}$
of $H_{\mathbb{C}}$. Let $t_0,\ldots,t_s$ be the coordinates on 
$H_{\mathbb{C}}$ associated to the basis.
We set $R=K[[t,q]]$ where $t=(t_0,t_{p+1},\ldots,t_s)$
and $q=(q_1,\ldots,q_p)$ with $q_i=e^{t_i}$.
Let $M^{\dagger}$ be the logarithmic formal scheme 
defined as in the previous paragraph. 

Fix a concave holomorphic vector bundle $\mathcal{V}$ on $X$ 
(e.g., $\mathcal{V}$ is the dual of an ample line bundle).   
We construct a formal MFS on $M^{\dagger}$ from $\mathcal{V}$
as follows. 
Let us introduce the fiberwise $S^1$-action on $\mathcal{V}$
by scalar multiplication. Then, 
following Givental \cite{Givental}, 
we consider
the $S^1$-equivariant Gromov--Witten invariants
of $X$ and the intersection pairing on $X$,
both twisted by
the inverse of the $S^1$-equivariant Euler class of $\mathcal{V}$.
Using them, one can define the twisted 
quantum cup product
$*_\mathcal{V}$ on $H^{\lambda}_{\mathbb{C}}:=H_{\mathbb C}\otimes_{\mathbb C} \mathbb{C}[\lambda]$
where $\mathbb{C}[\lambda]=H_{S^1}^*(pt, \mathbb{C})$ 
is identified with 
the $S^1$-equivariant cohomology of a point. 
Identifying the logarithmic tangent sheaf 
$\mathcal{T}^{\lambda}_{M^{\dagger}}:=\mathcal{T}_{M^{\dagger}}\otimes_{\mathbb{C}} \mathbb{C}[\lambda]$
of $M^{\dagger}$ over $\mathbb{C}[\lambda]$ with  
$\mathcal{O}_M \otimes_{\mathbb{C}} H^{\lambda}_{\mathbb{C}}$,  
we obtain a formal Frobenius structure  over $\mathbb{C}[\lambda]$ on $M^{\dagger}$.
Then, as an application of the results in \S \ref{sec:formal-MFS}, 
we obtain a formal MFS on $M^{\dagger}$ 
in the non-equivariant limit  (i.e. the limit as $\lambda\to 0$)
(Theorem \ref{thm:main}).

As we mentioned earlier, our motivation to study MFS 
comes from local mirror symmetry \cite{CKYZ}.
Relationships to {\it loc.cit.} and to our previous work \cite{KonishiMinabe10} 
are explained in \S\ref{sec:LMS}. 

\subsection{Conventions}
(1) Let $K$ be a field. A $K$-algebra means a finite-dimensional commutative associative $K$-algebra 
with a unit. 
\\
(2) Given a commutative ring $R$,
an $R$-algebra structure on a free $R$-module
means an associative commutative $R$-bilinear multiplication 
which admits a unit.

\subsection*{Acknowledgements}
The authors thank Hiroshi Iritani for discussion and 
various comments. Y.K. is supported in part by 
JSPS Grant-in-Aid for Young Scientists (No.\,22740013). 
S.M. is supported in part by 
JSPS Grant-in-Aid for Young Scientists (No.\,24740028). 

\section{Mixed Frobenius algebra}
\label{section:MFA}

\subsection{Frobenius filtration and mixed Frobenius algebra}
Let $K$ be a field.
A nondegenerate symmetric bilinear form $g$ on a $K$-vector space is called a {\it metric}.
A pair $(I_{\bullet},g_{\bullet})$ consisting 
of an exhaustive increasing filtration $I_{\bullet}$ on
a $K$-vector space by subspaces and
a collection of metrics $g_{\bullet}$ on $I_{\bullet}/I_{\bullet-1}$
is called a {\it nondegenerate filtration} on the vector space.

Let $A$ be a $K$-algebra.
We say that a metric $g$ on an $A$-module $I$ is $A$-invariant
if it satisfies the condition
$$
g(a \cdot x,y)=g(x,a\cdot y)\quad (a\in A,~x,y\in I)~.
$$ 

\begin{definition}\label{def:frobenius-filtration}
A Frobenius filtration on a $K$-algebra $A$ is a nondegenerate
filtration $(I_{\bullet},g_{\bullet})$ on $A$ such that 
each filter $I_{\bullet}$ is an ideal of $A$
and the metric $g_{\bullet}$ on $I_{\bullet}/I_{\bullet-1}$ is 
$A/I_{\bullet-1}$-invariant.
\end{definition}

\begin{definition}
A mixed Frobenius $K$-algebra is a pair which consists 
of a $K$-algebra $A$
and a Frobenius filtration $(I_{\bullet},g_{\bullet})$ on 
$A$.
\end{definition}

\subsection{Existence of Frobenius filtrations}
\label{MFA-existence}
In this subsection, the field $K$ is assumed to be algebraically closed.
\begin{theorem}\label{thm:frobenius}
Any finite-dimensional $K$-algebra $A$ has a Frobenius filtration.
\end{theorem}
Let $\mathfrak{N}=\sqrt{0}$ be the nilradical of $A$. Note that the finite-dimensionality of
$A$ implies that $\mathfrak{N}$ is a nilpotent ideal and 
that $\mathfrak{N}$ coincides with the Jacobson radical of $A$.   
It follows from the latter that $A/\mathfrak{N}$ is a semi-simple algebra. 
Consider the decreasing sequence of ideals
$A\supset \mathfrak{N}\supset \mathfrak{N}^2 \supset \cdots \supset \mathfrak{N}^{l-1} \supset \mathfrak{N}^l=0$.

\begin{lemma}
$\mathfrak{N}^i/\mathfrak{N}^{i+1}$ is a completely reducible $A/\mathfrak{N}^{i+1}$-module.
\end{lemma}
\begin{proof}
Consider the exact sequence of $A$-modules
$$0\longrightarrow \mathfrak{N}/\mathfrak{N}^{i+1} \longrightarrow A/\mathfrak{N}^{i+1} \longrightarrow A/\mathfrak{N}\longrightarrow0~.$$
It follows that $A/\mathfrak{N}^{i+1}$ acts on $\mathfrak{N}^i/\mathfrak{N}^{i+1}$ via $A/\mathfrak{N}$,  since $\mathfrak{N}/\mathfrak{N}^{i+1}$ 
annihilates $\mathfrak{N}^i/\mathfrak{N}^{i+1}$.
Then the semi-simplicity of $A/\mathfrak{N}$ implies that $\mathfrak{N}^i/\mathfrak{N}^{i+1}$ is a completely reducible $A/\mathfrak{N}$-module, 
hence it is also a completely reducible $A/\mathfrak{N}^{i+1}$-module.
\end{proof}

\begin{lemma}
Let $B$ be a finite-dimensional $K$-algebra. Then for any simple $B$-module $S\neq 0$, we have $\dim_{K}S=1$.
\end{lemma}
\begin{proof}
Since $S$ is a simple $B$-module,  there is a maximal ideal $\mathfrak{m}$ of $B$
such that $S\cong B/\mathfrak{m}$ as a $B$-module. 
The finite-dimensionality of $B$ implies that the composition 
$K\to B \to B/\mathfrak{m}$ is a field extension of finite degree.
It then follows that $B/\mathfrak{m}\cong K$, since $K$ is algebraically closed.   
\end{proof}

\begin{proof}[Proof of Theorem \ref{thm:frobenius}]
By the above two Lemmas, $\mathfrak{N}^i/\mathfrak{N}^{i+1}$ is the direct sum of 
$1$-dimensional simple $A/\mathfrak{N}^{i+1}$-modules. 
If we take a basis $x_{i,j}$ ($1\leq j\leq \mathrm{dim}_K\,\mathfrak{N}^i/\mathfrak{N}^{i+1}$)
of the simple modules and define a bilinear form $\langle~,~\rangle_i $
on $\mathfrak{N}^i/\mathfrak{N}^{i+1}$ by
$$
\langle x_{i,j},x_{i,k}\rangle_i=\delta_{j,k}~,
$$
then
$\langle~,~\rangle_i$ is an invariant metric.
Thus the filtration $I_{\bullet}:=\mathfrak{N}^{l-\bullet}$
with metrics $g_{\bullet}:=\langle~,~\rangle_{i-\bullet}$
is a Frobenius filtration on $A$.
\end{proof}

\section{Mixed Frobenius algebra from a localized $K[\lambda]$-metric}\label{section:MFA2}
\subsection{Construction of mixed Frobenius algebra}
\label{lambda-metric}
Let $K$ be a field and let $H_K$ be a $K$-vector space of 
dimension $s$. We set $H^{\lambda}_K:=H_K\otimes_K K[\lambda]$
and identify $H_K$ with the quotient module $H^{\lambda}_K/\lambda H^{\lambda}_K$. 
Let $\pi:H^{\lambda}_K\to 
H_K=H^{\lambda}_K/\lambda H^{\lambda}_K$ be the projection.

\begin{definition}
A  {\it localized $K[\lambda]$-metric} on $H^{\lambda}_K$ 
is a symmetric $K[\lambda]$-bilinear form 
$g^{\lambda}:H^{\lambda}_K \times H^{\lambda}_K \to K[\lambda,\lambda^{-1}]$
which is unimodular over $K[\lambda,\lambda^{-1}]$\footnote{
This means that, given a $K[\lambda]$-basis of $H_K^{\lambda}$, 
the representation matrix of $g^{\lambda}$ 
is a unimodular matrix over $K[\lambda,\lambda^{-1}]$.}.
\end{definition}


Now assume that a localized $K[\lambda]$-metric $g^{\lambda}$ on $H^{\lambda}_K$ is given. We will construct a nondegenerate filtration 
on $H_K$ from $g^{\lambda}$.

\begin{lemma}\label{lemma:elementary-divisors}
There exist
a pair of $K[\lambda]$-module bases
$\bm{x}_1,\ldots, \bm{x}_s$ and 
$\bm{y}_1,\ldots, \bm{y}_s$ of $H^{\lambda}_K$
and a set of integers $\kappa_1\geq \ldots\geq \kappa_s$
satisfying
\begin{equation}\label{elementary-divisors}
g^{\lambda}(\bm{x}_i,\bm{y}_j) =\lambda^{-\kappa_i}\delta_{i,j}~.
\end{equation}
The integers 
$\kappa_i$'s are uniquely determined by   $g^{\lambda}$
(but not the bases).
\end{lemma}
\begin{proof}
Let $G$
be the matrix representation of $g^{\lambda}$
 with respect to a $K[\lambda]$-module basis of $H^{\lambda}_K$.
 Multiplying by $\lambda^{k_0}$ with some $k_0\in \mathbb{Z}$
if necessary, we assume that
all entries of the matrix $\lambda^{k_0}G$ are polynomials.  
By the theorem of elementary divisors,
$\lambda^{k_0}G$ can be transformed into a diagonal matrix
by successive elementary transformations
from the left and the right. This means that 
there exist $K[\lambda]$-module bases 
$\{\bm{x}_i\}$, $\{\bm{y}_i\}$
of $H^{\lambda}_K$ such that $\lambda^{k_0}
g_{\lambda}(\bm{x}_i,\bm{y}_j)=
\delta_{i,j}\,e_i$
where $e_1,\ldots,e_s\in K[\lambda]$ are  diagonal entries
(i.e. the elementary divisors of $\lambda^{k_0}G$).
The assumption of unimodularity over $K[\lambda, \lambda^{-1}]$
implies that 
$e_i$'s are monomials.
\end{proof}

Let us define the sequence of $K[\lambda]$-submodules:
\begin{equation}\label{def-iklambda}
I_{k}^{\lambda}=\{\bm{x}\in H^{\lambda}_K \mid 
\lambda^k g^{\lambda}(\bm{x},\bm{y})\in K[\lambda]~
(\forall \bm{y}\in H^{\lambda}_K)\}~\quad (k\in \mathbb{Z})~.
\end{equation}
Concretely, $I_k^{\lambda}$ is written as follows with
the basis $\bm{x}_1,\ldots, \bm{x}_s$ of Lemma \ref{lemma:elementary-divisors}.
\begin{equation}\label{ik1}
I_k^{\lambda}=\bigoplus_{i: \kappa_i\leq k} K[\lambda]\,\bm{x}_i\oplus
\bigoplus_{i:\kappa_i>k}\lambda^{\kappa_i-k}K[\lambda]\,\bm{x}_i~.
\end{equation}
The same formula holds for the other basis $\{\bm{y}_i\}$.  

\begin{lemma}\label{g1}
For $\bm{x},\bm{y}\in I_k^{\lambda}$,
$\mathrm{Res}_{\lambda=0}\lambda^{k-1}g^{\lambda}(\bm{x},\bm{y})$
depends only on $\pi(\bm{x}),\pi(\bm{y})\in H_K$.
Moreover
$\mathrm{Res}_{\lambda=0}\lambda^{k-1}g^{\lambda}(\bm{x},\bm{y})=0$
if $\bm{x}\in I_{k-1}^{\lambda}$ or $\bm{y}\in I_{k-1}^{\lambda}$.
\end{lemma}

\begin{proof}
Let us write
$\bm{x}$, $\bm{y}\in I_k^{\lambda}$ as 
\begin{equation}\nonumber
\begin{split}
\bm{x}&=\sum_{i:\kappa_i\leq k} f_i(\lambda)\bm{x}_i+
\sum_{i:\kappa_i>k}\lambda^{\kappa_i-k}f_i(\lambda)\bm{x}_i 
\quad (f_i(\lambda)\in K[\lambda])~,
\\
\bm{y}&=\sum_{i:\kappa_i\leq k} h_i(\lambda)\bm{y}_i+
\sum_{i:\kappa_i>k}\lambda^{\kappa_i-k}h_i(\lambda)\bm{y}_i 
\quad (h_i(\lambda)\in K[\lambda]).
\end{split}
\end{equation}
By eq.\eqref{elementary-divisors}, we obtain
\begin{equation}\label{p1}
\mathrm{Res}_{\lambda=0}\lambda^{k-1}g^{\lambda}(\bm{x},\bm{y})=
\sum_{i:\kappa_i=k}f_i(0)h_i(0)~.
\end{equation}
The statement  follows easily from this.
\end{proof}

Let
\begin{equation}\label{ik2}
I_k:=\pi(I_k^{\lambda})=\bigoplus_{i: \kappa_i\leq k} K \pi(\bm{x}_i)
\quad (k\in \mathbb{Z})~.
\end{equation}
By Lemma \ref{g1},  the following bilinear form
$g_k$ on $I_k/I_{k-1}$ is well-defined:
\begin{equation}\label{gk}
g_k(\bar{x},\bar{y})=
\mathrm{Res}_{\lambda=0}\lambda^{k-1}g^{\lambda}(\bm{x},\bm{y})\quad
(x,y\in I_k)~,
\end{equation}
where $x\mapsto \bar{x}$ denotes the projection $H_K\to H_K/I_{k-1}$ and
$\bm{x},\bm{y}$ are any lifts of $x,y$ to 
$I_k^{\lambda}$. 
\begin{lemma}\label{induced-nondegenerate-filtration}
$(I_{\bullet},g_{\bullet})$ is a nondegenerate 
filtration on $H_K$.
\end{lemma}
\begin{proof} Nondegeneracy of $g_k$ follows from eq.\eqref{p1}.
\end{proof}

Now assume that $H^{\lambda}_K$ is equipped with
an associative commutative $K[\lambda]$-algebra structure $\ast$
with unit.
Let $g^{\lambda}$ be a localized $K[\lambda]$-metric
which is $\ast$-invariant, i.e., $g^{\lambda}$ satisfies
\begin{equation}\label{glambda-invariance}
g^{\lambda}(\bm{x}\ast \bm{y},\bm{ z})=
g^{\lambda}(\bm{x},\bm{y}\ast  \bm{z})\quad 
(\bm{x},\bm{y},\bm{z}\in H^{\lambda})~.
\end{equation}
On $H_K$,
we have the induced multiplication and
the nondegenerate filtration $(I_{\bullet},g_{\bullet})$ 
defined in eqs.\eqref{ik2}, \eqref{gk}. 
\begin{theorem}\label{thm-MFA}
$(H_K, I_{\bullet},g_{\bullet})$ is a mixed Frobenius algebra.
\end{theorem}
\begin{proof}
The $\ast$-invariance \eqref{glambda-invariance} of $g^{\lambda}$ implies that
$I_k^{\lambda}$ is an ideal.
Therefore $I_k$ is an ideal with respect to the induced multiplication
$\circ$ on $H_K$.
The $\circ$-invariance of $g_k$ 
follows from the $\ast$-invariance of $g^{\lambda}$.
\end{proof}

\subsection{Nilpotent construction}\label{nilpotent-construction}
Let $(A,g)$ be a Frobenius $K$-algebra having nilpotent elements 
$n_1, \ldots, n_r$ and let
$$\bm{n}=\lambda^r+n_1\lambda^{r-1}+\cdots + n_r\in A[\lambda].$$
As an example of Theorem \ref{thm-MFA},
we consider the case $H_K^{\lambda}=A[\lambda]$
with the localized $K[\lambda]$-metric $g^{\lambda}$
given by 
$$g^{\lambda}(\bm{x},\bm{y}):=
g\Big(\bm{x}\cdot \bm{y},~{\bm n}^{-1}\Big)
=\sum_{j\geq 0} \frac{1}{\lambda^{(j+1)r}} g(\bm{x}\cdot \bm{y},~{(\lambda^r-\bm n)}^j)
\quad (\bm{x},\bm{y}\in A[\lambda])~.
$$

Let us calculate the nondegenerate filtration $(I_{\bullet},g_{\bullet})$
defined by eqs.\eqref{ik2}, \eqref{gk}.
The ideals $I_k^{\lambda}$ of $A[\lambda]$ defined in eq.\eqref{def-iklambda}
are given as follows.

\begin{lemma}
\label{ex-1}
We have
\begin{equation}\begin{split}
I_{k}^{\lambda}=\begin{cases}
\{ \lambda^{-k} \bm n \cdot \bm{x}\mid \bm{x}\in A[\lambda]\} &(k\leq 0)\\
I_{0}^{\lambda}\oplus J_{k}^{\lambda}&(k>0)
\end{cases}~~.
\end{split}
\end{equation}
In the last line,  the direct sum is that of $A$-modules and
$$
J_k^{\lambda}=\{\bm{x}\in A^{<r}[\lambda]\mid
\lambda^k \bm{x} \text{ is divisible by $\bm{n}$ }\}~,
$$
where $A^{<r}[\lambda]=\{\bm{x}\in A[\lambda]\mid \mathrm{deg} \,\bm{x}<r\}$.
Here $\mathrm{deg} \,\bm{x}$ is the degree of $\bm x$ with respect to $\lambda$.
\end{lemma}
\begin{proof}
Since $\bm n$ is monic of degree $r$,
any $\bm x\in A[\lambda]$ is written uniquely as
${\bm x}={\bm n}\cdot{\bm x}' + {\bm x}''$ with ${\rm deg}\, {\bm x}''<r$.

First, we consider the case $k=0$.
It is easy to see that
$\bm x \in I_0^{\lambda}$ if and only if
$g^{\lambda}({\bm x}'', y)=0$ for any $y\in A$. 
If ${\bm x}''=\sum_{i=1}^{r}a_i\lambda^{r-i}$ then we have,
$$g^{\lambda}({\bm x}'', y)=\sum_{1\leq i\leq r, ~j\geq 0}
g(a_i{(\lambda^r-\bm n)}^j,y)\lambda^{-(i+jr)}
=\frac{g(a_1,y)}{\lambda}+o(\frac{1}{\lambda^2}).
$$
From this equation, 
for $x\in I^{\lambda}_0$, it is necessary
to have $g(a_1, y)=0$ for any $y\in A$, 
hence $a_1=0$. Then we have
$$g^{\lambda}({\bm x}'', y)=\frac{g(a_2,y)}{\lambda^2}+o(\frac{1}{\lambda^3}),$$
hence $a_2=0$. Repeating this process, we obtain ${\bm x}''=0$.

Next, we consider the case $k<0$. If $\bm x \in I_k^{\lambda}$ then $\bm x \in I_0^{\lambda}$.
Therefore it is written as ${\bm x}={\bm n}\cdot{\bm x}'$. 
Since ${\bm x} \in I_k^{\lambda}$, we have
$\lambda^k g^{\lambda}({\bm x}, y)=\lambda^kg({\bm x}', y) \in K[\lambda]$
for any $y\in A$. 
It then follows that the coefficients of ${\bm x}'$ up to degree $-k-1$
must be zero. 
Hence ${\bm x}'$ is divisible by $\lambda^{-k}$. 

For the case $k>0$,
it is easy to see that $\bm x \in I_k^{\lambda}$ if and only if $\lambda^k {\bm x}''$ is divisible 
by $\bm n$. 
\end{proof}

Let $N:A^{\oplus r}\to A^{\oplus r}$ 
be the homomorphism given by
$$
N\begin{pmatrix}a_1\\\vdots \\a_r\end{pmatrix}=
\begin{pmatrix}
-n_1&1&0&\cdots&0\\
-n_2&0&1&\cdots&0\\
\vdots&\vdots&\vdots&\ddots&\vdots\\
-n_{r-1}&0&0&\cdots&1\\
-n_r&0&0&\cdots&0
\end{pmatrix}
\begin{pmatrix}a_1\\\vdots \\a_r\end{pmatrix}~.
$$
The projections $A^{\oplus r}\to A$ to the first and the $r$th
factors are denoted $p_1$, $p_r$.
\begin{lemma}\label{ex-2}We have
\begin{equation}
I_k=\begin{cases}
0&(k<0)\\
\{x\cdot n_r \mid x\in A\}&(k=0)\\
I_{0}+ J_{k}&(k>0)
\end{cases}~~,
\end{equation}
where $J_k=p_r({\rm Ker}\, N^k)$.
\end{lemma}

\begin{proof}
By Lemma \ref{ex-1},
it is enough to show that 
$\pi(J_k^{\lambda})=J_k$ for $k>0$.
Let $\rho:A^{<r}[\lambda]\rightarrow A^{\oplus r}$
be the isomorphism
$\sum_{i=1}^r a_i \lambda^{r-i}\mapsto 
{}^t(a_1,\ldots,a_r)$.
Notice that 
$\rho^{-1}\circ N\circ \rho:A^{<r}[\lambda]\to A^{<r}[\lambda] $
maps $\bm{x}$
to the remainder of $\lambda \bm{x}$ divided by $\bm{n}$.
By induction on $k$, we can show that 
\begin{equation}\label{division}
\lambda^k \bm{x}
=\sum_{i=0}^{k-1}
(p_1\circ N^i \circ \rho)({\bm x})
\lambda^{k-1-i}\cdot \bm{n}
+(\rho^{-1}\circ N^k\circ \rho)({\bm x})
\quad (\bm{x}\in A^{<r}[\lambda])~.
\end{equation}
Thus  we obtain
$$
J_k^{\lambda}=\{\bm{x}\in A^{<r}[\lambda]\mid
\rho({\bm x})\in \mathrm{Ker}\,N^k
\}~.
$$
From this $\pi(J_k^{\lambda})=J_k$ follows.
\end{proof}

\begin{lemma}\label{ex-3}
We have
\begin{equation}\begin{split}\nonumber
&g_{0}(x\cdot n_r,y\cdot n_r)=g(x\cdot y,n_r)~,
\\
&g_k(\bar{x},\bar{y})
=g(x, p_1(N^{k-1}{\vec{y}}\,))\quad (k>0,~x,y\in J_k)~,
\end{split}\end{equation}
where $\vec{y}\in \mathrm{Ker}\,N^k$ is any lift of $y$
satisfying $p_r(\vec{y})=y$.
\end{lemma}
\begin{proof} The case $k=0$ is clear.
To show the case $k>0$, let ${\bm x},{\bm y}\in J_k^{\lambda}$ be lifts of 
$x,y$. By eq.\eqref{division}, we have
$$
g_k(\bar{x},\bar{y})=g^{\lambda}\Big({\bm x},\frac{\lambda^k{\bm y}}{\bm n}\Big)
\Big |_{\lambda=0}=g(x,p_1(N^{k-1}\rho(\bm{y})))~.
$$
\end{proof}

As a corollary of Theorem \ref{thm-MFA}, we obtain
\begin{proposition}
$(A,I_{\bullet},g_{\bullet})$ 
with $I_{\bullet}, g_{\bullet}$ given
in Lemmas \ref{ex-2}, \ref{ex-3},
is a mixed Frobenius algebra.
\end{proposition}

When $r=1$, 
$J_k=\{x\in A\mid n_1^k\cdot x=0\}$
and $g_k(\bar{x},\bar{y})=g(x\cdot y, (-n_1)^{k-1})$ ($k>0$).
This is 
the nilpotent construction in \cite[\S 3]{KonishiMinabe12}
up to shifts of the filtration.

\section{Mixed Frobenius structure}\label{sec:MFS}
In this section, the base field is $K=\mathbb{C}$, 
a manifold means a complex manifold and 
vector bundles are assumed to be holomorphic.
For a manifold $M$,
$T_M$ denotes the tangent bundle,
$\mathcal{T}_M$ its sheaf of local sections
and we write $x\in \mathcal{T}_M$ to mean that
$x$ is a local section of $T_M$. 

Although definitions here are for complex manifolds,
they can be easily translated to $C^{\infty}$-manifolds
($K=\mathbb{R}$).

\subsection{Saito structure}
The following definition
is due to Sabbah \cite[Ch.VII]{Sabbah}. 
\begin{definition}
\label{def-Saito-structure}
Let $M$ be a manifold. A Saito structure (without a metric) on $M$ 
consists of 
\begin{itemize}
\item a torsion-free flat connection $\nabla$ on $T_M$,
\item an associative and commutative $\mathcal{O}_M$-bilinear multiplication $\circ$ 
on $\mathcal{T}_M$ with a global unit section $e$, and
\item a global vector field $E$ on $M$ (called {\it the Euler vector field}),  
\end{itemize}
satisfying the following conditions.\\
(i) The multiplication $C_x$ by $x\in \mathcal{T}_M$ regarded as a local section 
of ${\rm End}\, T_M$
satisfies
\begin{equation}\label{mfs1}
\nabla_x C_y-\nabla_y C_x=C_{[x,y]}~,
\end{equation}
and the unit vector field $e$ is flat, i.e. $\nabla e=0$.\\
(ii) The vector field $E$ satisfies $\nabla (\nabla E)=0$ and
\begin{equation}\label{E1}
[E,x\circ y]-[E,x]\circ y-x\circ[E,y]=x\circ y\quad(x,y\in\mathcal{T}_M)~.
\end{equation}
\end{definition} 
In this article, we call a Saito structure without a metric
{\it  a Saito structure} for short.
\begin{remark}
In \cite{Sabbah},
a Saito structure is defined in terms of 
the symmetric Higgs field instead of the multiplication.
As explained in \cite[Ch.0.13]{Sabbah}, a symmetric Higgs field corresponds to
a multiplication and Definition \ref{def-Saito-structure} is equivalent to 
that in {\it loc.cit.}.
\end{remark}

\begin{lemma}\label{lemma:vector-potential}
Given a Saito structure $(\nabla,\circ,E)$ on a manifold $M$,  
there exists a local vector field $\mathcal{G}\in \mathcal{T}_M$
such that
\begin{equation}\label{eq:vector-potential}
\nabla_x\nabla_y \,\mathcal{G}=x\circ y 
\end{equation}
holds for any flat vector fields $x,y\in \mathcal{T}_M$.
Moreover
$\nabla\nabla([E,\mathcal{G}]-\mathcal{G})=0$.  
\end{lemma}
We call $\mathcal{G}$ satisfying eq.\eqref{eq:vector-potential}
a (local) potential vector field.

\begin{proof}
Let $\{t_{\alpha}\}_{\alpha}$ be a local coordinate system on $M$
whose corresponding local frame fields $\{ \partial_{\alpha}\}_{\alpha}$ are $\nabla$-flat.
Let us write $\partial_{\alpha}\circ \partial_{\beta}=
\sum_{\gamma} C_{\alpha\beta}^{\gamma} \partial_{\gamma}$.
The commutativity implies $C_{\alpha\beta}^{\gamma}=
C_{\beta\alpha}^{\gamma}$.
Eq.\eqref{mfs1} is equivalent to 
$\partial_{\alpha} C_{\beta\gamma}^{\delta}=
\partial_{\beta} C_{\alpha\gamma}^{\delta}$.
Therefore there exist $\mathcal{G}^{\gamma}\in \mathcal{O}_M$
such that  $\partial_{\alpha}\partial_{\beta} \mathcal{G}^{\gamma}=
C_{\alpha\beta}^{\gamma}$.
Then $\mathcal{G}:=\sum_{\gamma} \mathcal{G}^{\gamma}\partial_{\gamma}$
satisfies eq.\eqref{eq:vector-potential}.
The second statement follows from
eq.\eqref{E1}.
\end{proof}

\begin{remark}
It is shown that Frobenius manifold structure defined by 
Dubrovin \cite{Dubrovin} 
is equivalent to
a Saito structure with a metric \cite[Ch.VII, Prop.2.2]{Sabbah}.
In the case when $M$ is a Frobenius manifold,
eq.\eqref{mfs1} is equivalent to the potentiality condition,
and the gradient vector field of the potential function is
a potential vector field.
\end{remark}

\subsection{Mixed Frobenius structure}
\begin{definition}\label{def-MFS}
A mixed Frobenius structure (MFS) on a manifold $M$ consists of 
a Saito structure $(\nabla,\circ,E)$ together with
\begin{itemize}
\item an increasing sequence of subbundles $I_{\bullet}$ of $T_M$, and
\item metrics (i.e. nondegenerate symmetric $\mathcal{O}_M$-bilinear forms)
$g_{\bullet}$ on $\mathcal{I}_{\bullet}/\mathcal{I}_{\bullet-1}$, 
\end{itemize}
satisfying the following conditions.\\
(i) $(\circ,I_{\bullet},g_{\bullet})$ is
 a mixed Frobenius algebra structure on $T_{M}$,
 i.e.  $\mathcal{I}_{k}$ are ideals of $\mathcal{T}_{M}$ 
 and  all $g_k$'s are $\circ$-invariant.\\  
(ii) The subbundles $I_k$ ($k\in \mathbb{Z}$) are preserved by $\nabla$
and the metrics are compatible with $\nabla$, i.e. 
\begin{equation}
z g_k(\overline{x},\overline{y})=g_k(\overline{\nabla_z\, x},\overline{y})+
g_k(\overline{x},\overline{\nabla_z\,y})\quad 
(k\in\mathbb{Z},~z\in \mathcal{T}_M,~ x,y\in \mathcal{I}_k)~.
\end{equation}
Here $x\mapsto \overline{x}$ denotes the projection $\mathcal{I}_k\to
\mathcal{I}_k/\mathcal{I}_{k-1}$.
\\
(iii) The subbundles $I_k$ ($k\in\mathbb{Z}$) are preserved by
$[E,-]$ and
there exists a collection of numbers $\{D_k\in K\}_{k\in \mathbb{Z}}$ 
(called {\it charges}) such that
\begin{equation}\label{E2}
Eg_k(\overline{x},\overline{y})-g_k(\overline{[E,x]},\overline{y})
-g_k(\overline{x},\overline{[E,y]})=
(2-D_k)g_k(\overline{x},\overline{y})\quad
(k\in \mathbb{Z},~x,y\in \mathcal{I}_k)~.
\end{equation}
\end{definition}
A MFS with the trivial filtration $I_{\bullet}$  (i.e. $0\subset T_M$)
is the same as a Saito structure with a metric \cite{Sabbah}
and also the same as a Frobenius manifold structure \cite{Dubrovin}.

\begin{lemma}
If $(\nabla,\circ,E,I_{\bullet},g_{\bullet}) $ is a MFS on a manifold $M$,
then each $\mathcal{I}_k\subset \mathcal{T}_M$ $(k\in \mathbb{Z})$ is involutive.
\end{lemma}
\begin{proof}
The lemma follows from the condition that 
$I_k$ is preserved by the torsion free affine connection $\nabla$.
\end{proof}
As a consequence, 
there exists a system of flat local coordinate system
$\{t_{k\alpha}\}_{k\in\mathbb{Z},1\leq \alpha\leq \dim I_k/I_{k-1}}$
such that $\{t_{k\alpha}\}_{k\le l, 1\leq \alpha\leq \dim I_k/I_{k-1}}$
is a local coordinate system of leaves of $I_l$.

\begin{remark}\label{rem:definitions}
The definition of MFS in this article is different from that in our 
previous article \cite[Definition 6.2]{KonishiMinabe12} in a few points.

Firstly charges $\{D_k\}$ are allowed to take any values.
The merit is that any mixed Frobenius algebra has a MFS
(see Proposition \ref{Algebra-has-MFS2} below)
whereas the condition $D_k=D_0-k$ in the old definition is quite restrictive. 

Secondly the compatibility conditions of the multiplication with the connection and the Euler vector field
are strengthened as we adopt the Saito structure.
(Compare \cite[eqs.(6.2), (6.9)]{KonishiMinabe12} with eqs.\eqref{mfs1}, \eqref{E1}). 
The reason for this change is  
the existence of a local potential vector field (Lemma \ref{lemma:vector-potential})
and the flat meromorphic connection \cite{Sabbah} on the Saito structure.
We believe that
they may play important roles in formulating
the local mirror symmetry as equivalence of MFS's
(cf. \S\ref{sec:LMS}).
\end{remark}

\subsection{An algebra with a Frobenius filtration has a MFS}
Let $(A,I_{\bullet},g_{\bullet})$ be a mixed Frobenius algebra.
We assume that
$A=\oplus_{d\in \mathbb{Z}}A_d$ is a graded algebra
satisfying
$I_k=\oplus_d I_k\cap A_d$.
Moreover we assume that there exist
$\{D_k\in \mathbb{Z}\mid k\in \mathbb{Z},~I_k/I_{k-1}\neq 0\}$ such that 
$$
g_k(x,y)=0 \quad \text{unless}\quad |x|+| y|=D_k~.
$$
Here $|x|$ denotes the degree of $x\in A$.
Notice that any mixed Frobenius algebra 
satisfies this assumption with $A=A_0$ and $D_0=0$.

Let $\{e_{k\alpha}\mid k\in \mathbb{Z},1\leq \alpha\leq \dim I_k/I_{k-1}\}$ 
be a homogeneous basis of $A$
such that $\{e_{k\alpha}\mid k\leq l,1\leq \alpha\leq \dim I_l/I_{l-1}\}$ 
is a basis of $I_l$.
Let $\{t_{k\alpha}\}$ be the associated coordinates
of $A$.

\begin{proposition}\label{Algebra-has-MFS2}
The trivial connection $d$, the multiplication on $A$,
the vector field
$$
E=\sum_{k,\alpha}(1-|e_{ka}|)t_{ka}\partial_{ka}~,
$$
and the nondegenerate filtration $(I_{\bullet},g_{\bullet})$
form a MFS on $A$
of charges $\{D_k\}$.
\end{proposition}

\section{Formal Mixed Frobenius structure}\label{sec:formal-MFS}
In this section we will define a formal (and logarithmic) version of the MFS 
using 
\cite{Gross} as reference.

In this section,  the base field $K$ may be any subfield of $\mathbb{C}$.

\subsection{Notations}\label{sec:notation-formal}
Fix a positive integer $n\in \mathbb{Z}_{>0}$ and a non-negative integer $m\in \mathbb{Z}_{\geq 0}$.
We set
$R=
K[[t_1,\ldots,t_n,q_1,\ldots,q_m]]
$
and let
\begin{equation}\label{eq:monoid}
P=\{ \Phi(t, q)\in R \mid 
\exists d_1, \ldots, d_m \in \mathbb{Z}_{\geq 0}~{\rm s.t.}~
q_1^{-d_1}\cdots q_m^{-d_m} {\Phi}(t, q) \in R^{\times}
\}
\end{equation}
which is a submonoid of $R$. 
Let $M={\rm Spf}\, R$ 
be the formal completion of $K^{n+m}=\mathrm{Spec}\,K[t,q]$
at the origin and let $P_M$ be the 
constant sheaf on $M$ with a stalk $P$. 
Denote by $M^{\dagger}$ the formal scheme $M$ 
equipped with the logarithmic structure 
associated to ${P}_M \hookrightarrow \mathcal{O}_M$.

Let $\mathcal{T}_{M^{\dagger}}$ 
be the sheaf of logarithmic vector fields on $M^{\dagger}$
which is freely generated by $\frac{\partial}{\partial t_{1}},\ldots,
\frac{\partial}{\partial t_n}$ and
$q_1\frac{\partial}{\partial q_1},\ldots,q_m\frac{\partial}{\partial q_m}$
over $\mathcal{O}_M$. Namely,
if we let 
\begin{equation}\label{derivations}
H_K=
\bigoplus_{\alpha=1}^n K \frac{\partial}{\partial t_{\alpha}}\oplus
\bigoplus_{i=1}^m K q_{i}\frac{\partial}{\partial q_i}~,
\end{equation}
then we have $\mathcal{T}_{M^{\dagger}}=\mathcal{O}_M \otimes_K H_K$. 
Define a flat connection $\nabla$ on $\mathcal{T}_{M^{\dagger}}$ by 
$\nabla=d\otimes 1_{H_K}$. 
The Lie bracket $[~,~]$ satisfies
\begin{equation}
[x,y]=\nabla_x\, y-\nabla_y\, x~\quad (x,y\in \mathcal{T}_{M^{\dagger}})~.
\end{equation}

\subsection{Formal mixed Frobenius structure}
We keep the notations in \S \ref{sec:notation-formal}. 
\begin{definition}\label{def:formal-saito}
A formal Saito structure on $M^{\dagger}$ consists of
\begin{itemize}
\item an $\mathcal{O}_M$-bilinear
multiplication $\circ$ 
on $\mathcal{T}_{M^{\dagger}}$ 
and
\item an element $E\in \mathcal{T}_{M^{\dagger}}$
\end{itemize}
satisfying the following conditions.\\
(i) The multiplication $\circ$ is compatible with $\nabla$ in
the sense that
\begin{equation}\label{fmfs1}
\nabla_x(y\circ z)=\nabla_y (x\circ z)\quad (x,y,z\in H_K)~,
\end{equation}
and the unit element $e$ satisfies $\nabla e=0$, 
i.e. $e\in H_K$.\\
(ii) The element $E$ satisfies $\nabla_x\nabla_y \,E=0$\footnote{
If we write
$E=\sum_{\alpha=1}^n E_{\alpha}\partial_{t_{\alpha}}
+\sum_{i=1}^m E_{i}q_{i}\partial_{q_{i}}
$,
the condition $\nabla\nabla E=0$ implies that
$E_{\alpha}$ $(1\leq \alpha\leq n)$,
and $E_i$ $(1\leq i\leq m)$ are linear polynomials in $t$
and independent of $q$.
} for $x,y\in H_K$ and
\begin{equation}
[E,x\circ y]-x\circ [E,y]-[E,x]\circ y=x\circ y~ \quad (x,y\in \mathcal{T}_{M^{\dagger}})~.
\end{equation}
\end{definition}
If $(\circ, E)$ is a formal Saito structure on $M^{\dagger}$, then
as in Lemma \ref{lemma:vector-potential},
there exists $\mathcal{G}\in K[\log q_1,\ldots,\log q_m]\otimes_K
\mathcal{T}_{M^{\dagger}}$ satisfying eq.\eqref{eq:vector-potential}
for $x,y\in H_K$.

\begin{definition}
A formal mixed Frobenius structure on $M^{\dagger}$ consists of
\begin{itemize}
\item a formal Saito structure $(\circ, E)$ on $M^{\dagger}$, and
\item a nondegenerate filtration $(I_{\bullet},g_{\bullet})$
on $H_K$,
\end{itemize}
satisfying the following conditions.\\
(i) $\mathcal{I}_{k}
=\mathcal{O}_M \otimes_K I_k$ is an ideal
and $g_k$ extended $\mathcal{O}_M$-bilinearly to $\mathcal{I}_{k}$
is $\circ$-invariant, i.e.
\begin{equation}
g_k(x\circ y,z)=g_k(y,x\circ z)\quad (x\in 
\mathcal{T}_{M^{\dagger}},~
y,z\in \mathcal{I}_{k}
)~.
\end{equation}
(ii) $\mathcal{I}_{k}$  
is preserved by $[E,-]$, i.e.
$[E, x]\in \mathcal{I}_{k} 
$ ($x\in \mathcal{I}_{k}
$)
and
there exists a collection of numbers 
$\{D_k\in K\mid  k\in \mathbb{Z}, I_k/I_{k-1}\neq 0\}$,
called {\it charges}, 
satisfying
\begin{equation}\label{Eg}
Eg_k(\bar{x},\bar{y})-g_k(\overline{[E,x]},\bar{y})-g_k(\bar{x},\overline{[E,y]})=(2-D_k)g_k(\bar{x},\bar{y})
\quad (x,y\in \mathcal{I}_{k} 
)~.
\end{equation}
Here $x\mapsto\bar{x}$ denotes the projection 
$\mathcal{I}_{k} \to \mathcal{I}_{k}/ \mathcal{I}_{k-1}$.
\end{definition}

\begin{remark}[on the convergent case]\label{remark-convergence}
Let $(\circ, E)$ (resp.$(\circ,E,I_{\bullet},g_{\bullet})$) 
be a formal Saito structure (resp. a formal MFS)
on $M^{\dagger}$
and let $C_{\alpha\beta}^{\gamma}\in \mathcal{O}_M$
($1\leq \alpha,\beta,\gamma\leq n+m$)
denote the structure constants of  $\circ$
with respect to the basis $(x_1,\ldots,x_{n+m})=(\partial_{t_{1}},\ldots,\partial_{t_n}, q_1\partial_{q_1},\ldots,q_m\partial_{q_m})$.
If there exists an open neighborhood $U'$ of 
$0\in K^{n+m}=\mathrm{Spec}\,K[t,q]$ where all $C_{\alpha\beta}^{\gamma}$
converge,
then 
$(\nabla, \circ, E)$ (resp. $(\nabla,\circ, E,I_{\bullet},g_{\bullet})$)
is a Saito structure (resp. a MFS) on 
$U=U'\cap \{q_1\cdots q_m\neq 0\}$
with local flat coordinates $t_1,\ldots,t_n,\log q_1,\ldots,\log q_m$.
In the case when the filtration $I_{\bullet}$ is trivial,
then $(U,\circ, E,e,g_{\bullet})$ is a Frobenius manifold 
with logarithmic poles along the divisor $\{q_1\cdots q_m=0\}$
(see \cite{Reichelt} for the definition). 
\end{remark}

\subsection{Localized formal Frobenius structure over $K[\lambda]$}
\label{sec:localized-formal}
We still keep the notations in \S \ref{sec:notation-formal}
and use superscripts $\lambda$ for those
tensored with $K[\lambda]$: 
$\mathcal{O}_M^{\lambda}:=\mathcal{O}_M \otimes_K K[\lambda]$, 
$H^{\lambda}_K= H_K\otimes_K K[\lambda]$, 
and $\mathcal{T}_{M^{\dagger}}^{\lambda}:=\mathcal{T}_{M^{\dagger}}\otimes_K K[\lambda]=\mathcal{O}_M \otimes_{K}H_K^{\lambda}$.
We have 
a flat connection $\nabla$ on $\mathcal{T}_{M^{\dagger}}^{\lambda}$ 
defined by the $K[\lambda]$-linear extension
of that introduced in \S \ref{sec:notation-formal}. 

\begin{definition}\label{def:loc-formal}
A localized formal Frobenius structure over $K[\lambda]$
on $M^{\dagger}$ consists of 
\begin{itemize}
\item an $\mathcal{O}_M^{\lambda}$-bilinear  multiplication $\ast$
on $\mathcal{T}_{M^{\dagger}}^{\lambda}$, 
\item 
an element $\bm{E}\in \mathcal{T}_{M^{\dagger}}^{\lambda}$, and
\item a localized $K[\lambda]$-metric $g^{\lambda}$ on $H^{\lambda}_K$ 
\end{itemize}
satisfying the following conditions.\\
(i) The multiplication $\ast$ is compatible with $\nabla$ in the sense that
\begin{equation}\label{comp2}
\nabla_{\bm{x}}(\bm{y}\ast \bm{z})=
\nabla_{\bm{y}} (\bm{x}\ast \bm{z})
\quad (\bm{x},\bm{y},\bm{z}\in H^{\lambda}_K)~,
\end{equation}
and the unit $\bm{e}$ satisfies $\nabla \bm{e}=0$,  i.e. $\bm{e}\in H^{\lambda}_K$.
\\
(ii) The element $\bm{E}$ satisfies 
$\nabla_{\bm{x}}\nabla_{\bm{y}}\, \bm{E}=0$ for
$\bm{x},\bm{y}\in H^{\lambda}_K$ and
\begin{equation}
\label{EF1}
[\bm{E}^{\lambda},\bm{x}\ast \bm{y}]-
\bm{x}\ast[\bm{E}^{\lambda},\bm{y}]-[\bm{E}^{\lambda},\bm{x}]\ast\bm{y}=\bm{x}\ast \bm{y}
\quad(\bm{x},\bm{y}\in \mathcal{T}_{M^{\dagger}}^{\lambda})~,
\end{equation}
where $\bm{E}^{\lambda}
:=\bm{E}+\lambda\frac{\partial}{\partial \lambda}$.
\\
(iii)
$g^{\lambda}$, extended $\mathcal{O}_M^{\lambda}$-bilinearly to
$\mathcal{T}_{M^{\dagger}}^{\lambda}$, is $\ast$-invariant.
\\
(iv)
There exists $D\in K$ (called a charge) satisfying
\begin{equation}\label{EF2}
\bm{E}^{\lambda}g^{\lambda}(\bm{x},\bm{y})-
g^{\lambda}([\bm{E}^{\lambda},\bm{x}],\bm{y})-
g^{\lambda}(\bm{x},[\bm{E}^{\lambda},\bm{y}])=
(2-D)g^{\lambda}(\bm{x},\bm{y})
\quad (\bm{x},\bm{y}\in \mathcal{T}_{M^{\dagger}}^{\lambda})~.
\end{equation}
\end{definition}

\begin{proposition}\label{construction2}
Let
$(\ast,\bm{E},g^{\lambda})$ be a localized formal Frobenius 
structure over $K[\lambda]$
of charge $D$ on $M^{\dagger}$.
Let $\circ$ be 
the multiplication on $\mathcal{T}_{M^{\dagger}}$ 
induced by $\pi:\mathcal{T}_{M^{\dagger}}^{\lambda}\to \mathcal{T}_{M^{\dagger}}
=\mathcal{T}_{M^{\dagger}}^{\lambda}/\lambda \mathcal{T}_{M^{\dagger}}^{\lambda}$,
$E=\pi(\bm{E})$,
and let $(I_{\bullet},g_{\bullet})$ be 
the nondegenerate filtration
on $H_K$ induced from the localized $K[\lambda]$-metric
$g^{\lambda}$(see Lemma \ref{induced-nondegenerate-filtration}).
Then $(\circ,E,I_{\bullet},g_{\bullet})$
is a formal MFS  on $M^{\dagger}$ of charges $\{D_k=D-k\}$.
\end{proposition}
\begin{proof}
First,  the conditions (i) and (ii) of Definition \ref{def:formal-saito} 
follow from the conditions (i) and (ii) 
in Definition \ref{def:loc-formal} respectively.

The $\ast$-invariance of $g^{\lambda}$ implies that
$\mathcal{O}_M\otimes I_k^{\lambda}=:\mathcal{I}_{k}^{\lambda}\subset \mathcal{T}_{M^{\dagger}}^{\lambda}$ 
is an ideal
with respect to $\ast$, which in turn implies that
$\mathcal{I}_{k}$ is an ideal with respect to $\circ$.
It also implies 
the $\circ$-invariance of the metrics $g_{\bullet}$.

Eq.\eqref{EF2} implies that  the Lie bracket 
$[\bm{E}^{\lambda},-]$ preserves
$\mathcal{I}_{k}^{\lambda}$. From this it follows that $[E,-]$ preserves $\mathcal{I}_{k}$.
Eq.\eqref{EF2}  also implies the condition \eqref{Eg} as follows. 
For $x,y\in \mathcal{I}_{k}$, we have
\begin{equation}\nonumber
\begin{split}
Eg_k(\bar{x},\bar{y})&=\mathrm{Res}_{\lambda=0}\lambda^{k-1}
(k+\bm{E}^{\lambda})g^{\lambda}(\bm{x},\bm{y})
\\
&\stackrel{\eqref{EF2}}{=}k g_k(\bar{x},\bar{y})+
\mathrm{Res}_{\lambda=0}\lambda^{k-1}
\{g^{\lambda}([\bm{E}^{\lambda},\bm{x}],\bm{y})
+g^{\lambda}(\bm{x},[\bm{E}^{\lambda},\bm{y}])
+(2-D)g^{\lambda}(\bm{x},\bm{y})
\}
\\
&=(2-D+k)g_k(\bar{x},\bar{y})
+g_k([E,x],y)+g_k(x,[E,y])~,
\end{split}
\end{equation}
where $\bm{x},\bm{y}\in \mathcal{I}_{k}^{\lambda}$ are lifts of $x,y$.
\end{proof}

\section{Local quantum cohomology} 
In this section,  $K$ denotes either $\mathbb{R}$ or $\mathbb{C}$.
\subsection{Notations} 
Let $X$ be a smooth complex projective variety. 
Let $\mathcal{V}\to X$ be a  concave\footnote{A 
vector bundle $\mathcal{V}$ is concave if $H^0(C,f^*\mathcal{V})=0$
for any genus zero stable map $(f,C)$ to $X$ of non-zero degree.} 
vector bundle of rank $r$. 
Let $S^1={\mathrm U} (1)$ act on $\mathcal{V}$ by the scalar multiplication on the fiber.
The generator of the $S^1$-equivariant cohomology of 
a point is denoted $\lambda$.

Let $H_K:=H^{\mathrm{even}}(X,K)$.
We fix a basis $\{\phi_{1},\ldots,\phi_p\}$
of $H^2(X,\mathbb{Z})$  satisfying the condition
that $\int_C \phi_i\geq 0$ for any curve $C\subset X$.\footnote{
The existence of such a basis follows from the fact that 
the Mori cone $\overline{NE}_{\mathbb{R}}(X)$ 
of a smooth projective variety $X$
does not contain a straight line (see e.g.\cite[Corollary 1.19]{KollarMori}).
If $\sigma$ denotes the image of $\overline{NE}_{\mathbb{R}}(X)$
in $H_2(X,\mathbb{R})$,
the dual cone $\sigma^{\vee}
=\{x\in H^2(X,\mathbb{R})\mid
\langle x,y\rangle\geq 0,~y\in \sigma\}$ is of maximal dimension.
Therefore there exists an integral basis $\phi_1,\ldots\phi_p$ of   $H^2(X,\mathbb{R})$ such that  $\phi_i\in \sigma^{\vee}$.}
We also fix a homogeneous basis 
$\{\phi_0=1,\phi_1,\ldots,\phi_p,\phi_{p+1},\ldots,\phi_s\}$
of $H_K$.

Let $t_0,\ldots,t_s$ be the coordinates on $H_K$ associated to the basis.
We set $R=K[[t,q]]$ where $t=(t_0,t_{p+1},\ldots,t_s)$
and $q=(q_1,\ldots,q_p)$ with $q_i=e^{t_i}$.
As in \S \ref{sec:notation-formal}, 
we consider the formal scheme $M={\rm Spf}\, R$
with a fixed logarithmic structure defined by the monoid \eqref{eq:monoid}
and denote it by $M^{\dagger}$. 
We identify $H_K$ with the linear space
of derivations on $R$ defined in eq.\eqref{derivations} by 
\begin{equation}
\begin{cases}
\phi_{\alpha} \mapsto 
\frac{\partial}{\partial t_{\alpha}}&(\alpha=0,~p+1,\cdots, s)\\
\phi_i \mapsto q_i\frac{\partial}{\partial q_i}&(1\leq i\leq p)
\end{cases}.
\end{equation}
Hence $\mathcal{T}_{M^{\dagger}}=\mathcal{O}_M \otimes_K H_K$.
The same notations 
$\mathcal{O}_M^{\lambda}$, 
$H^{\lambda}_K$ and
$\mathcal{T}_{M^{\dagger}}^{\lambda}$
as in \S\ref{sec:localized-formal} will be used. 

We put the grading on the vector space $H_K$ by setting
$
|\phi|=k
$ if $\phi\in H^{2k}(X,K)$.
We also put the gradings on the rings $\mathcal{O}_M$ and $\mathcal{O}_M^{\lambda}$ by
$
|t_\alpha|=1-|\phi_{\alpha}| $ ($\alpha=0,p+1,\cdots, s$),
$|\lambda|=1$ and 
$|q_i|=\xi_i$, where $\xi_i$ are defined by 
\begin{equation}
c_1(X)+c_1(\mathcal{V})=\sum_{i=1}^p \xi_i \phi_i~.
\end{equation}
Then we have the induced gradings on
$\mathcal{T}_{M^{\dagger}}$ and $\mathcal{T}^{\lambda}_{M^{\dagger}}$.

Let 
\begin{equation}\label{Euler}
\bm{E}=\sum_{\alpha=0}^s (1-|\phi_{\alpha}|)
\,t_{\alpha}\frac{\partial}{\partial t_{\alpha}}+
\sum_{i=1}^p \xi_i q_i\frac{\partial}{\partial q_i}~,\quad
\bm{E}^{\lambda}=\bm{E}+\lambda\frac{\partial}{\partial\lambda}
~.\end{equation}
Then, for a homogeneous $f\in \mathcal{O}_{M}^{\lambda}$
and $\bm{x}\in \mathcal{T}_{M^{\dagger}}^{\lambda}$, 
we have
\begin{equation}\label{grading-Euler}
\bm{E}^{\lambda} f=|f|f~,\quad 
[\bm{E}^{\lambda},\bm{x}]=(|\bm{x}|-1)\bm{x}~.
\end{equation}
\subsection{Localized formal Frobenius structure over $K[\lambda]$}
The following material can be found in \cite{Givental}.
Let $g^{\lambda}$ be a localized $K[\lambda]$-metric on $H_K^{\lambda}$
defined by
\begin{equation}
g^{\lambda}(\phi,\varphi)=
\int_X \phi\cup\varphi\cup\frac{1}{e_{S^1}(\mathcal{V})}
\end{equation}
where $e_{S^1}(\mathcal{V})$ is the $S^1$-equivariant Euler class of $\mathcal{V}$:
$$e_{S^1}(\mathcal{V})=\lambda^r+c_1(\mathcal{V})\lambda^{r-1}+\cdots + c_r(\mathcal{V}).$$
\begin{lemma}
 $g^{\lambda}$ and $\bm{E}$ (in eq.\eqref{Euler})
satisfy eq.\eqref{EF2} with $D=\dim_{\mathbb{C}} X+r$.
\end{lemma}
\begin{proof}
By the degree consideration, $g^{\lambda}$ satisfies
\begin{equation}\label{eq:glambda}
g^{\lambda}(\phi_{\alpha},\phi_{\beta})=\eta_{\alpha \beta} \lambda^{|\phi_{\alpha}|+|\phi_{\beta}|-\dim_{\mathbb{C}} X-r}\quad
(\eta_{\alpha \beta}\in K)~.
\end{equation}
This together with eq.\eqref{grading-Euler} implies the lemma.
\end{proof}

We define the multiplication on $\mathcal{T}_{M^{\dagger}}^{\lambda}$
as follows.
For $x_1,\ldots,x_m\in H_K$
and ${d}\in H_2(X,\mathbb{Z})$,
let 
\begin{equation}
\langle x_1,\ldots,x_m \rangle_{\mathcal{V},{d}}
=\int_{[\overline{M}_{0,m}(X,{d})]^{\mathrm{vir}}}
\prod_{i=1}^m ev_i^*x_i\cup {e}_{S^1}
(-R^{\bullet}\mu_*ev_{m+1}^*\mathcal{V}) ~\in K[\lambda]
\end{equation}
where $\overline{M}_{0,m}(X,{d})$ is 
the moduli stack of genus zero stable maps to $X$ of degree ${d}$
with $m$ marked points,
$ev_i: \overline{M}_{0,m}(X,{d})\to X$ is the evaluation map
at the $i$th marked point, and 
$\mu: \overline{M}_{0,m+1}(X,{d})\to \overline{M}_{0,m}(X,{d})$
is the forgetful map. 
We define the multiplication $\ast_{\mathcal{V}}$ on $\mathcal{T}_{M^{\dagger}}^{\lambda}$ by
\begin{equation}\label{def-ast}
\begin{split}
g^{\lambda}(\bm{x}\ast_{\mathcal{V}} \bm{y},\bm{z})&=
\sum_{{d}}\sum_{m\geq 0}\frac{1}{m!}
\langle \bm{x},\bm{y},\bm{z},\underbrace{\tau, \ldots,\tau}_{m} 
 \rangle_{\mathcal{V},{d}} \quad 
 (\bm{x},\bm{y},\bm{z}\in \mathcal{T}_{M^{\dagger}}^{\lambda})
 \\
 &=\sum_{{d}}
 \sum_{m\geq 0}\frac{1}{m!}
 \langle \bm{x},\bm{y},\bm{z},
 \underbrace{\tau_{\ge 4}, \ldots,\tau_{\ge 4}}_{m} 
 \rangle_{\mathcal{V},{d}} \,q^{{d}}~.
\end{split}
\end{equation}
In the first line, $\tau=\sum_{\alpha=0}^s t_{\alpha}\phi_{\alpha}$
and in the second line,
$\tau_{\ge 4}=\sum_{\alpha=p+1}^s t_{\alpha}\phi_{\alpha}$ and
$q^{{d}}=e^{\int_{{d}} t_1\phi_1+\cdots+t_p \phi_p}$.
In passing to the second line, the fundamental class axiom and
the divisor axiom of Gromov--Witten theory 
(see, e.g., \cite[III, \S 5]{Manin}) are used.

\begin{lemma}
$(\mathcal{T}_{M^{\dagger}}^{\lambda},\ast_{\mathcal{V}})$ is a graded ring.
Hence the multiplication $\ast_{\mathcal{V}}$ and 
$\bm{E}$ in eq.\eqref{Euler} satisfy eq.\eqref{EF1}.
\end{lemma}
\begin{proof}
The lemma follows from the degree axiom of Gromov--Witten theory.
\end{proof}

\begin{proposition}\label{local1}
$(g^{\lambda},\ast_{\mathcal{V}}, \bm{E})$ is a localized formal Frobenius structure over 
$K[\lambda]$ of charge $\dim_{\mathbb{C}} X+r$ on $M^{\dagger}$.
\end{proposition}
\begin{proof}
By the definition of $\ast_{\mathcal{V}}$, it is clear that $g^{\lambda}$ is $\ast_{\mathcal{V}}$-invariant
and satisfies  eq.\eqref{comp2}.
\end{proof}

\subsection{Formal mixed Frobenius structure from local quantum cohomology}
\begin{theorem}\label{thm:main}
The collection $(\circ_{\mathcal{V}}, E, I_{\bullet}, g_{\bullet})$ 
of the following data determines a formal MFS of 
charges $\{\dim_{\mathbb{C}} X+r-k\}_{k\in \mathbb{Z}}$
on $M^{\dagger}$;
\begin{itemize}
\item the multiplication $\circ_{\mathcal{V}}$ on $\mathcal{T}_{M^{\dagger}}$ 
induced from the multiplication $\ast_{\mathcal{V}}$ on $\mathcal{T}_{M^{\dagger}}^{\lambda}$,
\item the Euler vector field $E$ which has the same expression 
as $\bm{E}$ in eq.\eqref{Euler},
\item a nondegenerate filtration $(I_{\bullet}, g_{\bullet})$ on $H_K$ 
constructed by Lemma \ref{induced-nondegenerate-filtration}.
\end{itemize}
\end{theorem}

\begin{proof}
Applying Proposition \ref{construction2}
to the localized formal Frobenius structure over $K[\lambda]$ in Proposition \ref{local1},
we obtain the result.
\end{proof}

\begin{remark}[on convergence of the formal MFS]
If $\mathcal{V}\to X$ is a negative line bundle,
it can be shown that
the structure constants of $\circ_{\mathcal{V}}$ are convergent
if those of  the quantum product of $X$ are convergent
e.g. if $X$ is a smooth projective toric variety \cite{Iritani}.
The proof is completely the same as Iritani's \cite{Iritani}
except that it is necessary to modify the proof of his Lemma 4.2.
For a pair of such $X$ and a negative line bundle $\mathcal{V}$, 
the formal MFS described in this subsection
is actually a MFS on some open subset of $H_K$.
(See Remark \ref{remark-convergence}).
\end{remark}

Let us describe the MFS in Theorem \ref{thm:main} concretely.
The multiplication $\circ_{\mathcal{V}}$ on $\mathcal{T}_{M^{\dagger}}$ 
is written as follows.
For ${d} \neq 0,~x_1,\ldots,x_m\in H_K$, let
\begin{equation}
\langle x_1,\ldots,x_m \rangle_{{\mathcal{V}},{d}}^{\lambda=0}
=\int_{[\overline{M}_{0,m}(X,{d})]^{\mathrm{vir}}}
\prod_{i=1}^m ev_i^*x_i\cup {e}
(R^{1}\mu_*ev_{m+1}^*\mathcal{V}) 
~,
\end{equation}
where $e$ denotes the
(non-equivariant) Euler class.
Then a potential vector field $\mathcal{G}$ for $\circ_{\mathcal{V}}$ 
(cf. Lemma \ref{lemma:vector-potential}) 
is given by
\begin{equation}\label{eq:G}
\mathcal{G}=\sum_{\alpha=0}^s
\Big(\partial_{\alpha} \Phi_{\rm cl} \Big)\,\phi^{\alpha}
+
\sum_{\alpha=1}^s
\Big(\partial_{\alpha} \Phi_{\rm qu} \Big)\,c_r(\mathcal{V})\cup \phi^{\alpha},
\end{equation}
where $\partial_{\alpha}=\frac{\partial}{\partial t_{\alpha}}$,
$$
\Phi_{\rm cl}=\dfrac{1}{3!}\int_X \tau\cup \tau\cup \tau,\qquad
\Phi_{\rm qu}=\sum_{{d}\neq 0}\sum_{m\geq 0}\frac{q^{{d}}}{m!}
\langle \underbrace{\tau_{\geq 4},\ldots,\tau_{\geq 4}}_{m}
\rangle_{\mathcal{V},{d}}^{\lambda=0},
$$
and $\{\phi^{\alpha}\}$ is a basis of $H_K$ dual to $\{\phi_{\alpha}\}$
with respect to the intersection form of $X$.

By the result of \S \ref{nilpotent-construction}, the nondegenerate filtration 
$(I_{\bullet}, g_{\bullet})$ on $H_K$ is given as follows. 
\begin{equation}\label{filtrationX}
\begin{split}
I_{k}&=0 \quad (k< 0)~,
\\
I_0&=\{x \cup c_r (\mathcal{V})\mid x\in H_K\}~,
\\
I_k&=I_0+J_k,\quad J_k=p_r(\mathrm{Ker}\, N^k),
\end{split}
\end{equation}
where 
$$
N=
\begin{pmatrix}
-c_1({\mathcal V})&1&0&\cdots&0\\
-c_2({\mathcal V})&0&1&\cdots&0\\
\vdots&\vdots&\vdots&\ddots&\vdots\\
-c_{r-1}({\mathcal V})&0&0&\cdots&1\\
-c_r({\mathcal V})&0&0&\cdots&0
\end{pmatrix}:~{H_K}^{\oplus r}\rightarrow {H_K}^{\oplus r}
$$
and $p_r$ is the projection to the $r$th factor.
The metrics $g_k$ on $I_k/I_{k-1}$ are given by
\begin{equation}\label{inner-productsX}
\begin{split}
g_0(c_r(\mathcal{V})\cup x,~c_r(\mathcal{V})\cup y)
&=\int_Xc_r(\mathcal{V})\cup x\cup y\quad (x,y\in H_K)~,
\\
g_k(\overline{x},~\overline{y})&=
\int_X x\cup p_1(N^{k-1}\vec{y}) \quad (k> 0, ~x,y\in J_k)~,
\end{split}
\end{equation}
where $\vec{y} \in \mathrm{Ker}\, N^{k} $ is any lift of $y$.

\begin{remark}[on the nilradical of $\circ_{\mathcal{V}}$]
If $(X,\mathcal{V})$ satisfies the condition that
$\int_C (c_1(X)+c_1({\mathcal{V}}))\leq 0$ holds for any curve $C\subset X$,
then 
$
\phi_{\alpha}\circ_{\mathcal{V}} \phi_{\beta}\in \mathcal{O}_M \otimes_{K}H^{\geq |\phi_{\alpha}|+|\phi_{\beta}|}(X,K)
$ holds by the degree axiom.
Therefore for such $(X,\mathcal{V})$,  the nilradical of $(\mathcal{T}_{M^{\dagger}},\circ_{\mathcal{V}})$ is
$\mathcal{O}_M \otimes_{K}H^{\ge 2}(X,K)$.
\end{remark}

\subsection{Remarks on local mirror symmetry}\label{sec:LMS}
Let $X$ be a Fano toric surface and $\mathcal{V}=K_X$ be the canonical bundle.
Take $\phi_{p+1}=\phi^0$. 
Then 
$$\mathcal{G}=\sum_{\alpha=0}^{p+1} (\partial_{\alpha}\Phi_{\rm cl}) \, \phi^{\alpha}
+\sum_{i=1}^p k_i(\partial_i\Phi_{\rm qu})\, \phi_{p+1}
$$
where
$$\Phi_{\rm qu} =\sum_{d \neq 0} N_{d}\, q^d,\qquad 
N_{d}=\int_{[\bar{M}_{0,0}(X, d)]^{\rm vir}}  e(R^{1}\mu_*ev_{m+1}^*K_X), $$
and $k_i$ are defined by
$\sum_{i=1}^p k_i\phi_i=c_1(K_X)$.  The coefficient of $\phi_{p+1}$
in the above $\mathcal{G}$
is nothing but the function
$\mathcal{F}_{local}$ in \cite[\S 6.3]{CKYZ}. 

Next, let us discuss the relationship with the mirror side of the story.  
Let $\Delta$ be the fan polytope of $X$.
There is a certain family of curves $\mathcal{C}\to \mathcal{M}(\Delta)$ in
$(\mathbb{C}^*)^2$ associated to $\Delta$.
It was shown that 
$$H^*(X, \mathbb{C}) \cong H^2((\mathbb{C}^*)^2, C_z)\qquad (z\in \mathcal{M}(\Delta))$$
as $\mathbb{C}$-vector spaces and that the weight filtration
of the mixed Hodge structure on  $H^2((\mathbb{C}^*)^2, C_z)$ coincides 
with Frobenius filtration (up to shifts). Compare 
\cite[\S 8]{KonishiMinabe10} with 
eq.\eqref{filtrationX} and \cite[eq.(8.8)]{KonishiMinabe12}.

Under the mirror map, $\mathcal{F}_{local}$ corresponds to a double logarithmic period
of $\omega_0(z)=[(\frac{dt_1}{t_1}\wedge \frac{dt_2}{t_2}, 0)]\in H^2((\mathbb{C}^*)^2, C_z)$
and 
$\{ g_0(\phi_i\circ_{K_X}\phi_j , c_1(K_X))\}_{1\leq i,j\leq p}$
is essentially equal to the Yukawa coupling defined in \cite[\S 6]{KonishiMinabe10}.

It is desirable to construct a MFS on  $H^2((\mathbb{C}^*)^2, C_z)$
which is compatible with its variation of mixed Hodge structures
and which agrees with the MFS on $H^*(X, \mathbb{C})$ under the mirror map.


\end{document}